\documentclass[11pt,reqno]{amsart}
\usepackage[left=2cm,right=2cm,bottom=3cm]{geometry}%
\usepackage{amsmath,amssymb}
\usepackage{cite}
\usepackage[pagewise]{lineno}
\numberwithin{equation}{section}

\usepackage{amsmath,amssymb,amsthm}

\usepackage{url}
\usepackage{here}
\usepackage{enumitem}
\newtheorem{theorem}{Theorem}[section]
\newtheorem{lemma}[theorem]{Lemma}
\newtheorem{proposition}[theorem]{Proposition}

\theoremstyle{definition}

\newtheorem{definition}[theorem]{Definition}

\newcommand{\ZZ}{\mathbb Z}
\newcommand{\RR}{\mathbb R}
\newcommand{\PP}{\mathbb P}

\newcommand{\QQ}{\mathbb Q}
\newcommand{\CC}{\mathbb C}

\newcommand{\Eff}{\overline{\operatorname{Eff}}}
\newcommand{\Mov}{\overline{\operatorname{Mov}}}

\newcommand{\Nef}{\operatorname{Nef}}

\newcommand{\Pic}{\operatorname{Pic}}

\newcommand{\rat}{\dashrightarrow}

\newcommand{\Bir}{\operatorname{Bir}}

\newcommand{\arrow}{\rightarrow}

\title{Remark on a theorem of Oguiso}

\author[I. Kaur]{Inder Kaur}

\address{School of Mathematics and Statistics, University of Glasgow, G12 8QQ , UK}

\email{inder.kaur@glasgow.ac.uk}

\author[A. Prendergast-Smith]{Artie Prendergast-Smith}

\address{Department of Mathematical Sciences, Loughborough University, LE11 3TU, UK}

\email{a.prendergast-smith@lboro.ac.uk}

\thanks{A.P-S. and I.K. were supported by EPSRC grant EP/W026554/1.}

\date{}
 
\begin{document}

\begin{abstract}
  For a Calabi--Yau variety $X$, Oguiso \cite{Oguiso2018} gave a useful criterion for primitivity of a self-map of $X$ in terms of the associated linear map on the N\'eron--Severi space of $X$. In this short note, we prove a variant of Oguiso's criterion and use it to verify primitivity of a certain birational automorphism of a Calabi--Yau threefold, to which Oguiso's original criterion does not apply.  
\end{abstract}

\maketitle

\section{Introduction}
The purpose of this note is to prove a variant of a criterion of Oguiso \cite{Oguiso2018} verifying that a birational self-map of a Calabi--Yau variety is primitive. Roughly speaking, a map is primitive if it does not ``factor through'' a birational automorphism of a lower-dimensional variety; in studying dynamics of birational maps, this is a natural condition to impose in order to exclude uninteresting examples such as product maps. The precise defintion of primitivity was first formulated by Zhang \cite{Zhang09}, and is stated by Oguiso \cite{Oguiso2018} as follows:

\begin{definition}
Let $X$ be a projective variety. A {\it rational fibration} means a dominant rational map $\pi \colon X \rat B$ where $B$ is a projective variety and $\pi$ has connected fibres. The fibration $\pi$ is {\it nontrivial} if $0< \dim B <\dim X$.

  Let $f \colon X \rat X$ be a birational map. A rational fibration $\pi \colon X \rat B$ is called {\it $f$-equivariant} if there is a birational map $f_B \colon B \rat B$ such that $\pi~\circ~f~=~f_B~\circ~\pi$. The birational map $f$ is called {\it primitive} if there does not exist a non-trivial $f$-equivariant rational fibration $\pi \colon X \rat B$.  
\end{definition}

For a Calabi--Yau variety $X$, Oguiso \cite[Theorem 1.2]{Oguiso2018} gave a useful criterion for primitivity of a birational map $f: X \rat X$ in terms of the associated linear map $f^*$ on the N\'eron--Severi space of $X$. The precise statement is Theorem \ref{theorem-oguisocriterion} below; roughly speaking, assuming general conjectures of minimal model theory, the criterion asks that $f^*$ should have no nontrivial invariant subspaces defined over $\QQ$. In Theorem \ref{theorem-oguisocriterionstrengthened}, we will prove a variant of Oguiso's criterion, replacing his linear algebraic criterion on $f^\ast$ by one involving convex geometry.

We will also be interested in the dynamical complexity of a birational map, as measured by Dinh--Sibony's notion of dynamical degrees \cite{DS05}. To keep our presentation simple, we only define the first dynamical degree of a birational map $f \colon X \rat X$, and moreover give a definition that is valid only in the special case when $f$ is an isomorphism in codimension 1.

\begin{definition}
  Let $f \colon X \rat X$ be a birational map which is an isomorphism in codimension 1, and let $f^\ast \colon N^1(X) \arrow N^1(X)$ be the linear map induced by pullback of divisors. The {\it first dynamical degree} $d_1(f)$ is the spectral radius of the map $f^\ast$, in other words,
  \begin{align*}
    d_1(f) = \max \left\{ \vert \lambda_i(f^\ast) \vert \right\}
  \end{align*}
  where the maximum is taken over the set of all eigenvalues $\{ \lambda_i(f^\ast)\}$ of the linear map $f^\ast$. 
\end{definition}
As motivation, in the case where $f \colon X \arrow X$ is biregular, the Gromov--Yomdin theorem \cite{Gr03,Yo87} says that the topological entropy of $f$ equals
\begin{align*}
  l(f) & := \operatorname{log}\, \operatorname{max} \left\{ d_i(f) \mid i = 1,\ldots, \dim X -1 \right\};
\end{align*} in the birational case Dinh-Sibony showed that $l(f)$ is an upper bound for topological entropy. So maps with $d_1(f)>1$ can be considered as candidates for having positive topological entropy.

We now introduce some notation and terminology in order to give precise statements of Oguiso's criterion and our variant of it.

We work throughout over $\CC$. For a $\QQ$-factorial projective variety $X$, we write $N^1(X)_\QQ$ to denote the vector space of divisors with rational coefficients modulo numerical equivalence, and $N^1(X)$ to denote $N^1(X)_\QQ \otimes \RR$. The {\it pseudoeffective cone} $\Eff(X)$ means the closed cone in $N^1(X)$ generated by the classes of effective divisors. A divisor $D$ on $X$ is {\it movable} if the intersection of all effective divisors in the complete linear system $|D|$ has codimension at least 2 in $X$. The {\it closed movable cone} $\Mov(X)$ means the closed cone in $N^1(X)$ generated by the classes of movable divisors. 

A projective variety $X$ is a {\it minimal Calabi--Yau variety} if $X$ is $\QQ$-factorial and terminal, we have $h^1(O_X)=0$, and the canonical divisor $K_X$ is trivial. We say a minimal Calabi--Yau variety is {\it m-abundant} if, for any movable effective divisor $D$, there is another minimal Calabi--Yau variety $X^\prime$ and a birational map $g \colon X^\prime \rat X$ such that $g^*D$ is semi-ample on $X^\prime$. By the existence of log minimal models and the log abundance theorem in dimension 3 \cite{Ka92, Sho03, KMM1994}, every minimal Calabi--Yau variety of dimension 3 is m-abundant.

We can now state Oguiso's criterion \cite[Theorem 1.2]{Oguiso2018}:
\begin{theorem}[Oguiso]\label{theorem-oguisocriterion}
  Let $X$ be a minimal Calabi--Yau variety of dimension at least $3$, of Picard number at least $2$, and which is m-abundant. Let $f \colon X \rat X$ be a birational map such that $f^*$ acts irreducibly on the $\QQ$-vector space $N^1(X)_\QQ$. Then $f$ is primitive. 
\end{theorem}

  
To state our variant, we fix some more terminology. Let $V$ be a vector space and $ g \colon V \arrow V$ a linear endomorphism. The {\it fixed subspace} of $g$ means the largest subspace $U \subset V$ such that $gu=u$ for all $u \in U$. A subspace $W \subset V$ is {\it $g$-stable} if $gw \in W$ for all $w \in W$. A face $F$ a cone $K$ is {\it proper} if $F\neq\{0\}$ and $F \neq K$. We will show the following:

\begin{theorem}\label{theorem-oguisocriterionstrengthened}
  Let $X$ be a minimal Calabi--Yau variety of dimension at least $3$, of Picard number at least $2$, and which is m-abundant. Let $f \colon X \rat X$ be a birational map such that:
  \begin{itemize}
  \item the fixed subspace of $f^*$ intersects the cone $\Eff(X)$ trivially;
  \item at least one of the following is true:
    \begin{itemize}
    \item  $\Mov(X)$ does not have a proper $f^*$-stable face defined over $\QQ$;
    \item $\Eff(X)$ does not have a proper $f^*$-stable face defined over $\QQ$.
    \end{itemize}
  \end{itemize}
  Then $f$ is primitive.
\end{theorem}

Section \ref{section-proofs} of this note will outline Oguiso's proof of Theorem \ref{theorem-oguisocriterion}, and explain how the proof can be modified to give the statement of Theorem \ref{theorem-oguisocriterionstrengthened}. In Section \ref{section-example} we will apply this variant of the criterion to verify primitivity of a certain birational automorphism to which Oguiso's original form of the criterion does not apply. We also show that the first dynamical degree of this map is strictly greater than 1.

\section{Outline proof of Theorem \ref{theorem-oguisocriterion} and proof of Theorem \ref{theorem-oguisocriterionstrengthened}} \label{section-proofs}
In this section we outline Oguiso's proof of Theorem \ref{theorem-oguisocriterion}. In particular we identify those points in the proof where the assumption of irreducibility of $f^\ast$ is used. We  then explain how to modify the proof so that it works under the weaker assumptions of Theorem \ref{theorem-oguisocriterionstrengthened}. In various places in the proof, we take $\nu \colon \widetilde{X} \arrow X$ to be a resolution of the singularities of $X$ and of the indeterminacy of $\pi \colon X \rat B$, take $\widetilde{\pi}  = \pi \circ \nu \colon \widetilde{X} \arrow B$, and define $\widetilde{X}_b$ to be the fibre of $\widetilde{\pi}$ over a point $b \in B$. 

{\bf Step 1:} The first step is to prove the weaker statement that, under the hypotheses of Theorem \ref{theorem-oguisocriterion}, there can be no nontrival $f$-equivariant fibration $\pi \colon X \rat B$ such that $\kappa(\widetilde{X}_b)=0$. Oguiso shows that, in this situation, we can assume after  replacing $X$ with a birational model $X^\prime$ that $\pi$ is a morphism, and that the nontrivial subspace $\pi^* N^1(B) \subset N^1(X)$ is stable for the action of $f^\ast$. This contradicts the assumption of irreducibility of $f^\ast$, showing that no such $\pi$ exists.

{\bf Step 2:} The next step is to show that if $P$ is a very general point of $X$, then the two-way orbit $\left\{ f^n(P) \mid n \in \ZZ \right\}$ is well-defined and is a Zariski-dense subset of $X$. As a corollary of a result of Lo Bianco \cite[Proposition 4.5.1]{LB19} also proved by Oguiso \cite[Proposition 2.5]{Oguiso2018}, this implies that for $\pi \colon X \rat B$ as before, the general fibre $\widetilde{X}_b$ is not of general type. The proof of the first assertion uses a result of Amerik--Campana \cite{AC08} which shows that there is a dominant rational map $\rho \colon X \rat C$ to a smooth projective variety $C$ such that $\rho \circ f = \rho$ and $\rho^{-1}( \rho (P))$ is the Zariski closure of the two-way orbit of $P$, for very general $P$. This implies that $\rho^* N^1(C)$ is a well-defined subspace of $N^1(X)$ on which $f^\ast$ acts as identity, again contradicting irreducibility of $f^\ast$ unless $N^1(C)$ is trivial, i.e. unless $C$ is a point and the orbit of $P$ is Zariski-dense.

{\bf Step 3:} The final step of the proof is to consider $\pi \colon X \rat B$ a nontrivial $f$-equivariant fibration, and to take the relative Iitaka fibration over $B$ to get $g \colon X \rat K$. By Step 2 we know that $\widetilde{X}_b$ does not have general type, implying that $\dim K < \dim X$; also by definition we have $\dim K \geq \dim B$. So this is again a nontrivial $f$-equivariant rational fibration. Moreover, by construction of the Iitaka fibration it has the key property that $\kappa(\widetilde{X}_k)=0$ for a general point $k \in K$. The existence of such a fibration contradicts the conclusion of Step 1, and so the proof is complete.

To generalise Oguiso's proof, we need to show that Steps 1 and 2 above still work under the weaker hypotheses of Theorem \ref{theorem-oguisocriterionstrengthened}. Let us deal with Step 2 first. Keeping the notation as above, we have the following:
\begin{lemma}
Suppose $C$ is not a point. For any nonzero basepoint-free divisor $D$ on $C$, we have that $\rho^*(D)$ is a nonzero effective divisor on $X$. In particular, the subspace $\rho^* N^1(C)$ intersects the cone $\Eff(X)$ nontrivially.
\end{lemma}
\begin{proof}
By definition $\rho^\ast=\nu_\ast \widetilde{\rho}^\ast D$ where $\nu \colon \widetilde{X} \arrow X$ is a resolution as before and $\widetilde{\rho} \colon \widetilde{X} \arrow C$ is the induced morphism. 
  
Now let $D$ be a basepoint-free divisor class on $C$. Then $\widetilde{\rho}^\ast D$ is basepoint-free on $\widetilde{X}$. In particular, we can choose an effecitve divisor in this class which is distinct from the union of all exceptional divisors of $\nu$. Then $\nu_\ast \widetilde{\rho}^\ast D$ is an effective and nonzero divisor on $X$. 
\end{proof}
We then get the conclusion of Step 2 above, under the weaker hypotheses of Theorem \ref{theorem-oguisocriterionstrengthened}. That is, if $f \colon X \rat X$ is a birational map such that the fixed subspace $f^*$ intersects the effective cone $\Eff(X)$ trivially, then for a very general point closed point $P \in X$ the points $f^n(P)$ are defined for all $n \in \ZZ$ and the two-way orbit $\left\{f^n(P) \mid n \in \ZZ \right\}$ is Zariski-dense in $X$.

Next we turn to Step 1. To adapt the proof to work under our weaker hypotheses, we note that in the above setup, the $f$-stable subspace $\pi^* N^1(B)$ contains an $f$-stable full-dimensional cone $K=\pi^* \Nef(B)$. The relative interior of $K$ consists of divisor classes which are semi-ample on the birational model $X^\prime$, and hence movable on $X$ itself. So if $B$ is not a point, the cone $K$ is a nonzero $f^*$-stable face of $\Mov(X)$ defined over $\QQ$. Moreover if $\dim B<\dim X$ then $X$ is covered by curves on which all divisors in $K$ have degree $0$, so no divisor whose class lies in $K$ can be big. Therefore in this case $K$ is also a proper $f^*$-stable face of $\Eff(X)$ defined over $\QQ$. Therefore, under the hypotheses of Theorem \ref{theorem-oguisocriterionstrengthened}, we conclude that there is no nontrivial $f$-equivariant fibration $\pi \colon X \rat B$ with $\kappa(\widetilde{X}_b)=0$.

Finally, Oguiso's argument in Step 3 does not use the assumption of irreducibililty of $f^\ast$, and so this step of the argument goes through unchanged. This completes the proof of Theorem \ref{theorem-oguisocriterionstrengthened}.

\section{Example}\label{section-example}
In this section, we give an example of a smooth Calabi--Yau variety $X$ of dimension 3 with a birational map $\varphi \colon X \rat X$ to which Oguiso's Theorem \ref{theorem-oguisocriterion} does not apply but Theorem \ref{theorem-oguisocriterionstrengthened} does. This criterion shows that $\varphi$ is a primitive birational map, and we will also see that it has first dynamical degree $d_1(\varphi)$ strictly greater than 1. 

Let $X$ be a general complete intersection of 3 hypersurfaces of degree $(1,1,1)$ in $\PP := \PP^2 \times \PP^2 \times \PP^2$. Bertini's theorem shows that $X$ is smooth, and the Lefschetz hyperplane theorem shows that $X$ is simply connected, in particular $H^1(O_X)=0$, and $H^2(O_X)=0$ also. By adjunction we have $K_X = (K_\PP)_{|X} \otimes \operatorname{det} N_{X/\PP} = O_X$. So $X$ is a smooth Calabi--Yau variety. Let $\pi_i \colon X \arrow \PP^2$ denote projection onto the $i$-th factor, and let $L_i=\pi_i^\ast(H)$ where $H$ is the class of a line in $\PP^2$. Note that $\Pic(X) = \bigoplus_{i=1}^3 \ZZ \cdot L_i$.

The following proposition gives the basic geometric properties that we need for the fibres of the morphisms $\pi_i$. The proofs are straightforward but tedious, so we defer them to the end of this section.

\begin{proposition}\label{prop-goodfibres}
The fibres of $\pi_i$ are 1-dimensional. For each $i$, there is an open set $U_i \subset \PP^2$ such that $\PP^2 \setminus U_i$ consists of finitely many points, and for $p \in U_i$, the fibre $\pi_i^{-1}(p)$ is reduced and irreducible.
\end{proposition}
By adjunction, the smooth fibres of each of the maps $\pi_i \colon X \arrow \PP^2$ are elliptic curves. For each $i$, let $X_{\eta_i}$ denote the generic fibre of the morphism $\pi_i$. We have
\begin{align*}
  \Pic(X_\eta) \cong \Pic(X)/\operatorname{Vert}(\pi_i)
\end{align*}
where $\operatorname{Vert}(\pi_i)$ denote the subgroup of $\Pic(X)$ spanned by effective divisors $D$ such that $\pi_i(D) \neq \PP^2$. By Proposition \ref{prop-goodfibres} any such $D$ is a multiple of $L_i$, so $\Pic(X_\eta) \cong \Pic(X)/L_i$.

Let $E_{ij}$ denote the restriction of the line bundle $L_j$ to $X_{\eta_i}$. For $j \neq i$ we have
\begin{align*}
  L_j \cdot L_i^2 \cdot (L_1+L_2+L_3)^3&=3
\end{align*}
so $E_{ij}$ is a line bundle of degree 3 on $X_{\eta_i}$.

Now let $i,j,k$ be any ordering of the indices $1,2,3$. Then the line bundle $E_{ij}-E_{ik}$ has degree $0$ on $X_{\eta_i}$. In general for a curve $C$ and $y \in \Pic^0(C)$, translation by $y$ acts on $\Pic(C)$ by the formula
\begin{align*}
  x &\mapsto x + \left( \operatorname{deg} x \right) y.
\end{align*}
In particular taking $C=X_{\eta_i}$ and $y=E_{ij}-E_{ik}$ we have
\begin{align*}
  E_{ij} \mapsto 4E_{ij} - 3E_{ik}\\
  E_{ik} \mapsto 3E_{ij} - 2E_{ik}
\end{align*}
The translation action of $\Pic^0(X_{\eta_i})$ on $X_{\eta_i}$ extends to a birational action on $X$. We denote by $\varphi_{ijk} \colon X \rat X$ the birational map corresponding to translation by $E_{ij}-E_{ik}$. Since $X$ is smooth and $K_X$ is trivial, by \cite[Theorem 3.52]{KM98} the map $\varphi_{ijk}$ in fact extends to a pseudo-automorphism of $X$ over $\PP^2$,  that is, a birational automorphism which is an isomorphism in codimension 1 and preserves the fibration $\pi_i$.

Fix $i=1$, $j=2$, $k=3$. By the previous displayed equations the linear map $(\varphi_{123})_*$ on $N^1(X)$ is represented by a matrix of the form
\begin{align*}
  M_{123} &= 
  \begin{pmatrix}
    1 & m & n\\
    0 & 4 & 3\\
    0 & -3 & -2
  \end{pmatrix}
\end{align*}
for some integers $m,n$.

If now $M_{ijk}$ represents the linear map $(\varphi_{ijk})_*$ we observe that on one hand, $M_{123}^{-1} =M_{132}$,  while on the other hand, $M_{132} = T_{23}M_{123}T_{23}$ where $T_{23}$ is the permutation matrix corresponding to the transposition $(23)$. This implies that $m=2n$.

To determine the missing integer $m$, we can proceed as follows. Let $H$ denote a general line in $\PP^2$, and let $S=\pi_1^{-1}(H)$. Then $S$ is a smooth surface. The map $\varphi_{123}$ preserves $H$ and hence restricts to an automorphism of $S$, which we denote by $\varphi^S$. Denote the restriction of the line bundle $L_i$ to $S$ by $\Lambda_i$. For $i=1,2,3$ we have $\varphi^S(\Lambda_i) = \varphi_{123}(L_i)_{|S}$. So in particular we have $\varphi^S(\Lambda_2)=m\Lambda_1+4\Lambda_2-3\Lambda_3$. Since $\varphi^S$ is an automorphism of a smooth surface, it preserves intersection numbers, so we get
\begin{align*}
 (m\Lambda_1+4\Lambda_2-3\Lambda_3)^2 &= \Lambda_2^2=3
\end{align*}
hence $m=12$. Since $m=2n$ this implies $n=6$, so we have
\begin{align*}
  M_{123} &= 
  \begin{pmatrix}
    1 & 12 & 6\\
    0 & 4 & 3\\
    0 & -3 & -2
  \end{pmatrix}
\end{align*}
Identical arguments give that $(\varphi_{231})_\ast$ and $(\varphi_{312})_\ast$ are represented respectively by the matrices
\begin{align*}
  M_{231} =
    \begin{pmatrix}
     -2 &0& -3 \\
      6&1& 12 \\
     3 &0&4 
    \end{pmatrix},&\
    M_{312} =
    \begin{pmatrix}
      4&3&0\\
      -3&-2&0\\
      12&6&1
    \end{pmatrix}.
\end{align*}
After this paper was completed, we discovered that these matrices had been computed previously by Hoff--Stenger--Y\'{a}\~{n}ez \cite[Example 4.2]{HSY22} by a different method.

Finally, the birational map of $X$ that we are interested in is
\begin{align*}
  \varphi &\colon X \rat X\\
  \varphi &= \varphi_{123} \circ \varphi_{231} \circ \varphi_{312}.
\end{align*}
The pullback map $\varphi^\ast \colon N^1(X) \arrow N^1(X)$ is then represented by the matrix
\begin{align*}
  M &= \left( M_{123} M_{231} M_{312} \right)^{-1}\\
  &= \begin{pmatrix}
    -44 & -330 & -615\\
    60 & 451 & 840\\
    165 & 1230 & 2296
  \end{pmatrix}
\end{align*}
We can now verify the claimed properties of our example:
\begin{proposition}
  The birational map $\varphi \colon X \rat X$ is primitive with first dynamical degree $d_1(\varphi) >1$.
\end{proposition}
\begin{proof} Using the matrix $M$ above we compute that the characteristic polynomial of $\varphi^\ast$ is
\begin{align*}
  \chi(\varphi^\ast)(t) &= \det \left( \varphi^\ast - t \cdot \operatorname{Id} \right) \\
  &= 1-2703t+2703t^2-t^3
\end{align*}
and the eigenvalues are
\begin{align*}
  \lambda_1 = 1, \ \lambda_{2,3} = 1351 \pm 780 \sqrt3.
\end{align*}
In particular we get that the first dynamical degree $d_1(\varphi)$ of the birational map $\varphi \colon X \rat X$ is
  \begin{align*}
d_1(\varphi) &=  1351 + 780 \sqrt 3.
  \end{align*}
  It remains to prove that $\varphi$ is a primitive birational map; for this we use Theorem \ref{theorem-oguisocriterionstrengthened}. Note that since $f^\ast$ has the rational eigenvalue $\lambda_1=1$, it does not act irreducibly on $N^1(X)_\QQ$ and so Theorem \ref{theorem-oguisocriterion} does not apply.

To show that Theorem \ref{theorem-oguisocriterionstrengthened} applies in our example, we first verify the condition concerning the fixed subspace of $\varphi^\ast$. The fixed subspace is of $\varphi^\ast$ is $1$-dimensional, spanned by the divisor class
\begin{align*}
  D_{fixed} &= L_1-2L_2+L_3.
\end{align*}

\begin{lemma}
Let $r$ be a nonzero real number. Then $rD_{fixed} \notin \Eff(X)$.
\end{lemma}
\begin{proof}
 First consider a divisor class $D=aL_1+bL_2+cL_3$ with $a,b,c \in \RR$. We claim that if any two of $a,b,c$ are strictly negative, then $D \notin \Eff(X)$. To see this, suppose for simplicity that $a<0$ and $b<0$. Let $C_3$ be any fibre of $\pi_3 \colon X \arrow \PP^2$. Then  $L_1 \cdot C_3>0$ and $L_2 \cdot C_3>0$, while $L_3 \cdot C_3=0$. So we have
  \begin{align*}
    D \cdot C_3 &= aL_1\cdot C_3 + bL_2\cdot C_3\\
    <0.
  \end{align*}
  On the other hand, since curves algebraically equivalent to $C_3$ cover $X$, if $\Delta$ is any effective divisor on $X$ we have $\Delta \cdot C_3 \geq0$. The same inequality therefore holds for any class $\Delta \in \Eff(X)$, showing that such a class $D$ cannot belong to $\Eff(X)$. 

  Now consider the class $D_{fixed}=L_1-2L_2+L_3$. If $r<0$, then $rD_{fixed}$ has two negative coefficients, so by the previous paragraph $rD_{fixed} \notin \Eff(X)$. If $r>0$, then $rD_{fixed} \in \Eff(X)$ if and only if $D_{fixed} \in \Eff(X)$. Since the cone $\Eff(X)$ is preserved by birational automorphisms, this implies that $\varphi_* D_{fixed} \in \Eff(X)$ for every $\varphi \in \Bir(X)$. But using the matrix $M_{123}$ above we compute that $(\varphi_{123})_* D_{fixed} = -17L_1-L_2+4L_3$, which by the previous paragraph cannot be in $\Eff(X)$. 
\end{proof}
This proves that the birational map $\varphi$ satisfies the first condition of Theorem \ref{theorem-oguisocriterionstrengthened}. To prove the second criterion, we can argue as follows. A 1-dimensional $\varphi^\ast$-stable face of $\Eff(X)$ would span a 1-dimensional eigenspace, but we have already seen that the unique nontrivial $\QQ$-eigenspace is the fixed subspace, which does not span a face of $\Eff(X)$ say. So we can restrict our attention to stable faces of dimension 2. If $F$ is a 2-dimensional $\varphi^\ast$-stable face of $\Eff(X)$, then its linear span is a 2-dimensional stable subspace $V_F \subset N^1(X)$. Consider the dual linear map $\left(\varphi^\ast\right)^\vee$ on $\left(N^1(X)\right)^\vee=N_1(X)$. The annhilator $V_F^\perp$ of $V_F$ would then be a 1-dimensional eigenspace of $\left(\varphi^\ast \right)^\vee$ defined over $\QQ$; moreover, the pseudoeffective cone $\Eff(X)$ is dual to the cone of nef curves, so this 1-dimensional eigenspace would contain the class of a nef curve. 

Computation shows that the only such eigenspace is spanned by the vector $v=L_1^\vee-2L_2^\vee+L_3^\vee$, where $\{L_1^\vee,L_2^\vee,L_3^\vee\}$ is the dual basis to $L_1,L_2,L_3$. Any nonzero multiple of $v$ is therefore negative on at least one of the effective divisors $L_1, L_2, L_3$, so this eigenspace cannot contain the class of a nef curve. By the last paragraph, this shows that there is no 2-dimensional $\varphi^\ast$-stable face of $\Eff(X)$ defined over $\QQ$.
\end{proof}
To finish, we give the proofs of the properties of the fibrations $\pi_i$ that were claimed in Proposition \ref{prop-goodfibres}. Recall that $X$ is a general complete intersection of 3 hypersurfaces of degree $(1,1,1)$ in $\PP^2 \times \PP^2 \times \PP^2$. To fix notation say that $F_1,\, F_2, \, F_3 \in H^0\left(\PP^2 \times \PP^2 \times \PP^2, O(1,1,1)\right)$ and
\begin{align*}
  X &= \left\{ \left([X_0,X_1,X_2],[Y_0,Y_1,Y_2],[Z_0,Z_1,Z_2]\right) \mid F_\alpha(X_i,Y_j,Z_k)=0 \text{ for } \alpha,\, i,\, j,\, k=1,2,3\right\}.
\end{align*}
For $i=1,2,3$, let $\pi_i \colon X \arrow \PP^2$ denote projection onto the $i$-th factor. To keep notation simple we write the required proofs for the case $i=1$; identical proofs work for $i=2$ and $i=3$. 
\begin{proposition}
The fibres of $\pi_1$ are 1-dimensional.  
\end{proposition}
\begin{proof}
  Choose $p \in \PP^2$ and denote by $F_i^p$ the form obtained by substituting homogeneous coordinates of $p$ in place of $X_0,X_1,X_2$ in the forms $F_i$. So the fibre $\pi_1^{-1}(p)$ equals the intersection $F_1^p \cap F_2^p \cap F_3^p$.
  
  First we claim that, for all $p \in \PP^2$, the vector space $\operatorname{span}\left( F_1^p ,\, F_2^p, \, F_3^p \right) \subset H^0(\PP^2 \times \PP^2, O(1,1))$ has dimension 3. To see this, write $F_i = \sum \Lambda_{ijk} Y_j Z_k$ for linear forms $\Lambda_{ijk}$ in the variables $X_0,X_1,X_2$. Then 
  \begin{align*}
    \operatorname{dim} \operatorname{span} \left( F_1^p ,\, F_2^p, \, F_3^p \right) &= \operatorname{rank}
    \begin{pmatrix} \Lambda_{100} & \Lambda_{110} &\cdots &\Lambda_{122} \\
      \Lambda_{200} & \Lambda_{210} &\cdots &\Lambda_{222}\\
         \Lambda_{300} & \Lambda_{310} &\cdots &\Lambda_{322}
    \end{pmatrix}(p)
  \end{align*}
  Since the $\Lambda_{ijk}$ are general, the locus of points $p \in \PP^2$ where this matrix drops rank is empty.

Next, since $\operatorname{span}\left( F_1^p ,\, F_2^p, \, F_3^p \right)$ has dimension 3, $F_3^p$ say does not vanish identically on $F_1^p \cap F_2^p$. So the only way that the fibre can have dimension 2 is if all of the $F_i^p$ vanish on a surface $S \subset \PP^2 \times \PP^2$ which is a component of a reducible complete intersection of the zero-loci of two sections of $O(1,1)$. A standard dimension count shows that the subset $R \subset Gr(3,V)$ parametrising linear systems whose base locus contains such a component has codimension 7 in $Gr(3,V)$. 

  The linear system $\operatorname{span}(F_1,F_2,F_3)$ defines a morphism
  \begin{align*}
    \phi \colon \PP^2 &\arrow Gr(3,V) \\
    p &\mapsto \operatorname{span}(F_1^p,F_2^p,F_3^p)
  \end{align*}
  where $V=V=H^0(\PP^2 \times \PP^2, O(1,1))$. The group $GL(9)$ acts transitively on $Gr(3,V)$, so by Kleiman's theorem the general translate of $\phi(\PP^2)$ is transverse to the codimension-7 subset $R$ defined above, hence disjoint from it. By computation we see that the action of $GL(9)$ on $V$ corresponds to changing the choice of forms $F_1,F_2,F_3$. So the general translate equals $\phi(\PP^2)$ for a suitable choice of forms $F_1,F_2,F_3$. Hence, for a general choice of forms $F_1,F_2,F_3$, the fibres of $\pi_1$ are all 1-dimensional.
\end{proof}


\begin{lemma}\label{lemma-codim2}
  Let $M$ denote the parameter space of complete intersection curves in $\PP^2 \times \PP^2$ which are defined by the vanishing of 3 sections of $O(1,1)$. The locus of curves which are reducible or generically non-reduced has codimension 2 in $M$. 
\end{lemma}
\begin{proof}[Sketch of proof]
  The variety $M$ is an open subset of the Grassmannian $Gr(3,V)$. The dimension of $V$ equals 9, so $M$ has dimension $18$.
  A curve parameterised by a general point of $M$ has bidegree $(3,3)$. Now suppose that $F$ is a reducible curve corresponding to a point of $M$; for simplicity assume it has 2 components. Then we have
  \begin{align*}
    F &= C_{i,j} \cup C_{3-i,3-j}
  \end{align*}
  where $C_{l,m}$ denotes a smooth rational curve of bidegree $(l,m)$ in $\PP^2 \times \PP^2$ and the intersection of the two components is a zero-dimensional scheme of length 2. We claim that for each possible type of reducible curve, the space of such curves has dimension at most 16.

  We give full details in the case $F=C_{2,3} \cup C_{1,0}$; other cases are similar. Let $pr_i \colon \PP^2 \times \PP^2 \arrow \PP^2$ denote the projections to the two factors. The space of smooth rational curves of bidegree $(2,3)$ has dimension 16. Let $C$ be such a curve: then $pr_2(C)$ is a cubic in $\PP^2$, which must be rational and therefore has a singular point $p$. Now let $C'$ be the component of bidegree $(1,0)$. It is contained in a fibre $pr_2^{-1}(q)$ for some $q$. But $C \cap C'$ has length 2 and $pr_2$ maps $C$ isomorphically onto its image away from $pr_2^{-1}(p)$, so we must have $C' \subset pr_2^{-1}(p)$. If $pr_2(C)$ has a node at $p$ then $C'$ must be the line joining the two points of $pr_2^{-1}(p) \cap C$; if $pr_2(C)$ has a cusp at $p$ then $C'$ must be the tangent line to $C$ at the point $pr_2^{-1}(p) \cap C$. In both cases $C'$ is uniquely determined by $C$, so the space of such curves has dimension $16$.\end{proof}
Now we can complete the proof of Proposition \ref{prop-goodfibres}.
\begin{proposition}[=Proposition \ref{prop-goodfibres}]
 The fibres of $\pi_1$ are irreducible and generically reduced in codimension 1. 
\end{proposition}
\begin{proof}
Again we consider the morphism
  \begin{align*}
    \phi \colon \PP^2 &\arrow M \subset Gr(3,9) \\
    p &\mapsto \operatorname{span}(F_1^p,F_2^p,F_3^p)
  \end{align*}
  defined by our choice of 3 forms. Let $S$ denote the subset of $M$ parametrising curves which are reducible or generically non-reduced. Lemma \ref{lemma-codim2} shows that $S$ has codimension 2 in $M$. As before, the general $GL(9)$-translate of $\phi(\PP^2)$ is transverse to $R$ and therefore $R \cap \phi(\PP^2)$ has codimension 2 in $\PP^2$. Again, the general translate equals $\phi(\PP^2)$ for a suitable choice of forms $F_1,F_2,F_3$.  
\end{proof}

{\bf{CONFLICT OF INTEREST}}: The authors have no conflict of interest.


\end{document}